\documentclass[12pt]{article}
\usepackage{amssymb,amsmath,amsthm,verbatim,mathdots}

\newtheorem{claim}{Claim}
\newtheorem{theorem}{Theorem}
\newtheorem{lemma}[theorem]{Lemma}

\newtheorem{question}[theorem]{Question}
\newtheorem{corollary}[theorem]{Corollary}
\theoremstyle{remark}
\newtheorem*{remark}{Remark}
\DeclareMathOperator{\diam}{diam}
\DeclareMathOperator{\dist}{dist}
\begin{document}
\title{On regular hypergraphs of high girth}

\author{David Ellis\footnote{Research supported in part by a Feinberg Visiting Fellowship from the Weizmann Institute of Science.}\\
\small School of Mathematical Sciences\\[-0.8ex]
\small Queen Mary, University of London\\[-0.8ex] 
\small UK\\
\small\tt d.ellis@qmul.ac.uk\\
\and
Nathan Linial\\
\small School of Computer Science and Engineering\\[-0.8ex]
\small Hebrew University of Jerusalem\\[-0.8ex]
\small Israel\\
\small\tt nati@cs.huji.ac.il
}


\date{\footnotesize{1st November 2013 (updated 12th July 2017)}}
\maketitle

\begin{abstract}
We give lower bounds on the maximum possible girth of an $r$-uniform, $d$-regular hypergraph with at most $n$ vertices, using the definition of a hypergraph cycle due to Berge. These differ from the trivial upper bound by an absolute constant factor (viz., by a factor of between $3/2+o(1)$ and $2 +o(1)$). We also define a random $r$-uniform `Cayley' hypergraph on the symmetric group $S_n$ which has girth $\Omega (n^{1/3})$ with high probability, in contrast to random regular $r$-uniform hypergraphs, which have constant girth with positive probability.
\end{abstract}

\section{Introduction}
The {\em girth} of a finite graph $G$ is the shortest length of a cycle in $G$. (If $G$ is acyclic, we define its girth to be $\infty$.) The {\em girth problem} asks for the minimum possible number of vertices $n(g,d)$ in a $d$-regular graph of girth at least $g$, for each pair of integers $d,g \geq 3$. Equivalently, for each pair of integers $n,d \geq 3$ with $nd$ even, it asks for a determination of the largest possible girth $g_{d}(n)$ of a $d$-regular graph on at most $n$ vertices.

The girth problem has received much attention for more than half a century, starting with Erd\H{o}s and Sachs \cite{erdossachs}. A fairly easy probabilistic argument shows that for any integers $d,g \geq 3$, there exist $d$-regular graphs with girth at least $g$. An extremal argument due to Erd\H{o}s and Sachs \cite{erdossachs} then shows that there exists such a graph with at most
$$2\frac{(d-1)^{g-1}-1}{d-2}$$
vertices. This implies that
\begin{equation}\label{eq:erdos-sachs}g_{d}(n) \geq (1- o(1)) \log_{d-1}n.\end{equation}
(Here, and below, $o(1)$ stands for a function of $n$ that tends to zero as $n \to \infty$.)

On the other hand, if $G$ is a $d$-regular graph of girth at least $g$, then counting the number of vertices of $G$ of distance less than $g/2$ from a fixed vertex of $G$ (when $g$ is odd), or from a fixed edge of $G$ (when $G$ is even), immediately shows that
$$|G| \geq n_0(g,d) := \left\{\begin{array}{lll} 1+d\sum_{i=0}^{k-1}(d-1)^i & = 1+d\frac{(d-1)^{k}-1}{d-2} & \textrm{if }g =2k+1;\\ 2\sum_{i=0}^{k-1}(d-1)^i & = 2\frac{(d-1)^{k}-1}{d-2} & \textrm{if }g=2k.\end{array}\right.$$
This is known as the {\em Moore bound}. Graphs for which the Moore bound holds with equality are known as {\em Moore graphs} (for odd $g$), or {\em generalized polygons} (for even $g$). It is known that Moore graphs only exist when $g=3$ or $5$, and generalized polygons only exist when $g = 4,6,8$ or $12$. It was proved in \cite{bannai,biggs,kovacs} that if $d \geq 3$, then
$$n(g,d) \geq n_0(g,d)+2\quad \textrm{for all } g \notin \{3,4,5,6,8,12\};$$
even for large values of $g$ and $d$, no improvement on this is known.

A related problem is to give an explicit construction of a $d$-regular graph of girth $g$, with as few vertices as possible. The celebrated Ramanujan graphs constructed by Lubotzsky, Phillips and Sarnak \cite{lps}, Margulis \cite{margulis} and Morgenstern \cite{morgenstern} constituted a breakthrough on both problems, implying that
\begin{equation}\label{eq:lps} g_{d}(n) \geq (4/3 - o(1))\log_{d-1}n\end{equation}
via an explicit (algebraic) construction, whenever $d =q+1$ for some odd prime power $q$.

One can obtain from this a lower bound on $g_{d}(n)$ for arbitrary $d \geq 3$, by choosing the minimum $d' \geq d$ such that $d' - 1$ is an odd prime power, taking a $d'$-regular Ramanujan graph with girth achieving (\ref{eq:lps}), and removing $d'-d$ perfect matchings in succession. This yields 
\begin{equation}\label{eq:lps2} g_{d}(n) \geq (4/3 - o(1))\frac{\log(d-1)}{\log(d'-1)} \log_{d-1}n.\end{equation}

In \cite{luw1} and \cite{luw2}, Lazebik, Ustimenko and Woldar give different explicit constructions (also algebraic), which imply that 
$$g_{d}(n) \geq (4/3 - o(1))\log_{d}n$$
whenever $d$ is an odd prime power, implying (\ref{eq:lps2}) whenever $d-1$ is not an odd prime power. (In fact, their constructions provide the best known upper bound on $n(g,d)$ for many pairs of values $(g,d)$.) Combining (\ref{eq:lps2}) with the Moore bound gives
\begin{equation}\label{eq:upper} (4/3 - o(1))\frac{\log(d-1)}{\log(d'-1)} \log_{d-1}n \leq g_{d}(n) \leq (2+o(1)) \log_{d-1}n.\end{equation}
Improving the constants in (\ref{eq:upper}) seems to be a very hard problem.

In this paper, we investigate an analogue of the girth problem for $r$-uniform hypergraphs, where $r \geq 3$. There are several natural notions of a cycle in a hypergraph. We refer the reader to Section~\ref{sec:open} for a brief discussion of some other interesting notions of girth in hypergraphs, and to \cite{duke} for a detailed treatise. Here, we consider the least restrictive notion, originally due to Berge (see for example \cite{berge} and \cite{berge2}). 

A {\em hypergraph} $H$ is a pair of finite sets $(V(H),E(H))$, where $E(H)$ is a family of subsets of $V(H)$. The elements of $V(H)$ are called the {\em vertices} of $H$, and the elements of $E(H)$ are called the {\em edges} of $H$. A hypergraph is said to be {\em $r$-uniform} if all its edges have size $r$. It is said to be {\em $d$-regular} if each of its vertices is contained in exactly $d$ edges. It is said to be {\em linear} if any two of its edges share at most one vertex.

Let $u$ and $v$ be distinct vertices in a hypergraph $H$. A {\em u-v path} of {\em length} $l$ in $H$ is a sequence of distinct edges $(e_1,\ldots,e_l)$ of $H$, such that $u \in e_1$, $v \in e_l$, $e_i \cap e_{i+1} \neq \emptyset$ for all $i \in \{1,2,\ldots,l-1\}$, and $e_i \cap e_j = \emptyset$ whenever $j > i+1$ (Note that some authors call this a  {\em geodesic path}, and use the term {\em path} when non-consecutive edges are allowed to intersect.) The {\em distance} from $u$ to $v$ in $H$, denoted $\textrm{dist}(u,v)$, is the shortest length of a {\em u-v} path in $H$. (We define $\dist(v,v) = 0$.) The {\em ball of radius $R$ and centre $u$} in $H$ is the set of vertices of $H$ with distance at most $R$ from $u$. The {\em diameter} of a hypergraph $H$ is defined by
$$\diam(H) = \max_{u,v \in V(H)} \dist(u,v).$$

A hypergraph is said to be a {\em cycle} if it has at least two edges, and there is a cyclic ordering of its edges, $(e_1,\ldots,e_l)$ say, such that there exist distinct vertices $v_1,\ldots,v_l$ with $v_i \in e_i \cap e_{i+1}$ for all $i$ (where we define $e_{l+1} := e_1$). This notion of a hypergraph cycle is originally due to Berge, and is sometimes called a {\em Berge-cycle}. The {\em length} of a cycle is the number of edges in it. The {\em girth} of a hypergraph is the length of the shortest cycle it contains.

Observe that two distinct edges $e,f$ with $|e \cap f| \geq 2$ form a cycle of length 2 under this definition, so when considering hypergraphs of high girth, we may restrict our attention to linear hypergraphs.

We use the Landau notation for functions: if $F,G:\mathbb{N} \to \mathbb{R}^{+}$, we write $F = o(G)$ if $F(n)/G(n) \to 0$ as $n \to \infty$. We write $F = O(G)$ if there exists $C>0$ such that $F(n) \leq CG(n)$ for all $n$. We write $F = \Omega(G)$ if there exists $c>0$ such that $F(n) \geq cG(n)$ for all $n$. Finally, we write $F = \Theta(G)$ if $F = O(G)$ and $F = \Omega(G)$.

Extremal questions concerning Berge-cycles in hypergraphs have been studied by several authors. For example, in \cite{bg}, Bollob\'as and Gy\H{o}ri prove that an $n$-vertex, 3-uniform hypergraph with no 5-cycle has at most $\sqrt{2} n^{3/2} + \tfrac{9}{2}n$ edges, and they give a construction showing that this is best possible up to a constant factor. In \cite{lv}, Lazebnik and Verstra\"ete prove that a 3-uniform, $n$-vertex hypergraph of girth at least 5 has at most
$$\tfrac{1}{6}n \sqrt{n-\tfrac{3}{4}} + \tfrac{1}{12}n$$
edges, and give a beautiful construction (based on the so-called `polarity graph' of the projective plane $\textrm{PG}(2,q)$) showing that this is sharp whenever $n = q^2$ for an odd prime power $q \geq 27$. Interestingly, neither of these two constructions are regular.

In \cite{gl} and \cite{thesis}, Gy\"ori and Lemons consider the problem of excluding a cycle of length exactly $k$, for general $k \in \mathbb{N}$. In \cite{gl}, they prove that an $n$-vertex, 3-uniform hypergraph with no $(2k+1)$-cycle has at most $4k^2 n^{1+1/k}+O(n)$ edges. In \cite{thesis}, they prove that an $n$-vertex, $r$-uniform hypergraph with no $(2k+1)$-cycle has at most $C_{k,r}(n^{1+1/k})$ edges, and furthermore that an $n$-vertex, $r$-uniform hypergraph with no $(2k)$-cycle has at most $C_{k,r}'(n^{1+1/k})$ edges, where $C_{k,r}, C_{k,r}'$ depend upon $k$ and $r$ alone.

In this paper, we will investigate the maximum possible girth of an $r$-uniform, $d$-regular hypergraph on $n$ vertices, for $r$ and $d$ fixed and $n$ large. If $r \geq 3$ and $d \geq 2$, we let $g_{r,d}(n)$ denote the maximum possible girth of an $r$-uniform, $d$-regular hypergraph on at most $n$ vertices. Similarly, if $d \geq 2$ and $r,g \geq 3$, we let $n_r(g,d)$ denote the minimum possible number of vertices in an $r$-uniform, $d$-regular hypergraph with girth at least $g$. Since a non-linear hypergraph has girth 2, we may replace `hypergraph' with `linear hypergraph' in these two definitions.

In section 2, we will state upper and lower bounds on the function $g_{r,d}(n)$, which differ by an absolute constant factor. The upper bound is a simple analogue of the Moore bound for graphs, and follows immediately from known results. The lower bound is a hypergraph extension of a similar argument for graphs, due to Erd\H{o}s and Sachs \cite{erdossachs} --- not a particularly difficult extension, but still, in our opinion, worth recording.

In section 3, we consider the girth of certain kinds of random $r$-uniform hypergraph. We define a random $r$-uniform `Cayley' hypergraph on $S_n$ which has girth $\Omega (n^{1/3})$ with high probability, in contrast to random regular $r$-uniform hypergraphs, which have constant girth with positive probability. We conjecture that, in fact, our `Cayley' hypergraph has girth $\Omega(n \log n)$ with high probability. We believe it may find other applications.

\section{Upper and lower bounds}
In this section, we state upper and lower bounds on the function $g_{r,d}(n)$, which differ by an absolute constant factor.

We first state a very simple analogue of the Moore bound for linear hypergraphs. For completeness, we give the proof, although the result follows immediately from known results, e.g. from Theorem 1 of Hoory \cite{hoorybipartite}.
\begin{lemma}
Let $r,d$ and $g$ be integers with $d \geq 2$ and $r,g \geq 3$. Let $H$ be an $r$-uniform, $d$-regular, $n$-vertex hypergraph with girth $g$. If $g=2k+1$ is odd, then
\begin{equation}\label{eq:odd}n \geq 1+d(r-1)\sum_{i=0}^{k-1} ((d-1)(r-1))^i = 1+d(r-1)\frac{(d-1)^{k}(r-1)^{k}-1}{(d-1)(r-1)-1},\end{equation}
and if $g=2k$ is even, then
\begin{equation}\label{eq:even}n \geq r\sum_{i=0}^{k-1} ((d-1)(r-1))^i = r\frac{(d-1)^{k}(r-1)^{k}-1}{(d-1)(r-1)-1}.\end{equation}
\end{lemma}
\begin{proof}
The right-hand side of (\ref{eq:odd}) is the number of vertices in any ball of radius $k$. The right-hand side of (\ref{eq:even}) is the number of vertices of distance at most $k-1$ from any fixed edge $e \in H$.
\end{proof}
The following corollary is immediate.
\begin{corollary}
\label{corr:upperbound}
Let $r,d$ and $g$ be integers with $d \geq 2$ and $r,g \geq 3$. Let $H$ be an $r$-uniform, $d$-regular hypergraph with $n$ vertices and girth $g$. Then
$$g \leq \frac{2\log n}{\log(r-1)+\log(d-1)}+2.$$
Hence,
$$g_{r,d}(n) \leq  \frac{2\log n}{\log(r-1)+\log(d-1)}+2.$$
\end{corollary}

Our aim is now to obtain a hypergraph analogue of the non-constructive lower bound (\ref{eq:erdos-sachs}). We first prove the following existence lemma.

\begin{lemma}
\label{lemma:lift}
For all integers $d \geq 2$ and $r,g \geq 3$, there exists a finite, $r$-uniform, $d$-regular hypergraph with girth at least $g$.
\end{lemma}
\begin{proof}
We prove this by induction on $g$, for fixed $r,d$. When $g=3$, all we need is a linear, $r$-uniform, $d$-regular hypergraph. Let $H$ be the hypergraph on vertex-set $\mathbb{Z}_r^d$, whose edges are all the axis-parallel lines, i.e. 
$$E(H) = \{\{\mathbf{x},\mathbf{x}+\mathbf{e}_i,\mathbf{x}+2\mathbf{e}_i,\ldots,\mathbf{x}+(r-1)\mathbf{e}_i\}:\ \mathbf{x} \in \mathbb{Z}_r^d,\ i \in [d]\}.$$
(Here, $\mathbf{e}_i$ denotes the $i$th standard basis vector in $\mathbb{Z}_r^d$, i.e. the vector with $1$ in the $i$th coordinate and zero elsewhere. As usual, $\mathbb{Z}_r$ denotes the ring of integers modulo $r$.) Clearly, $H$ is linear and $d$-regular. 

For $g \geq 4$ we do the induction step. We start from a finite, linear, $r$-uniform, $d$-regular hypergraph $H$ of girth at least $g-1$. Of all such hypergraphs we consider one with the least possible number of $(g-1)$-cycles. Let $M$ be the number of $(g-1)$-cycles in $H$. We shall prove that $M=0$. If $M>0$, we consider a random 2-lift $H'$ of $H$, defined as follows. Its vertex set is $V(H') = V(H) \times \{0,1\}$, and its edges are defined as follows. For each edge $e \in E(H)$, choose an arbitrary ordering $(v_1,\ldots,v_r)$ of the vertices in $e$, flip $r-1$ independent fair coins $c^{(1)}_e,\ldots,c_e^{(r-1)} \in \{0,1\}$, and include in $H'$ the two edges
$$\{(v_1,j),(v_2,j \oplus c^{(1)}_e),\ldots,(v_r,j\oplus c^{(r-1)}_e)\}\mbox{~~for~~}j=0,1.$$
(Here, $\oplus$ denotes modulo 2 addition.) Do this independently for each edge. Note that $H'$ is linear and $d$-regular, since $H$ is.

Let $\pi:V(H') \to V(H)$ be the cover map, defined by $\pi((v,j))=v$ for all $v \in V(H)$ and $j \in \{0,1\}$. Since any cycle in $H'$ is projected to a cycle in $H$ of the same length, $H'$ has girth at least $g-1$, and each $(g-1)$-cycle in $H'$ projects to a $(g-1)$-cycle in $H$. Let $C$ be a $(g-1)$-cycle in $H$. We claim that $\pi^{-1}(C)$ either consists of two vertex-disjoint $(g-1)$-cycles in $H'$, or a single $2(g-1)$-cycle in $H'$, and that the probability of each is $1/2$. To see this, let $(e_1,\ldots,e_{g-1})$ be any cyclic ordering of $C$; then $|e_i \cap e_{i+1}|=1$ for all $i$ (since $H$ is linear). Let $e_i \cap e_{i+1} = \{w_i\}$ for all $i \in [g-1]$. For each $i$, consider the two edges in $\pi^{-1}(e_i)$. Either one of the two edges contains $(w_{i-1},0)$ and $(w_{i},0)$ and the other contains $(w_{i-1},1)$ and $(w_{i},1)$, {\em or} one edge contains $(w_{i-1},0)$ and $(w_i,1)$ and the other edge contains $(w_{i-1},1)$ and $(w_i,0)$. Call these two events $S(e_i)$ and $D(e_i)$, for `same' and `different'. Observe that $S(e_i)$ and $D(e_i)$ each occur with probability 1/2, independently for each edge $e_i$ in the cycle. Notice that $\pi^{-1}(C)$ consists of two disjoint $(g-1)$-cycles if and only if $D(e_i)$ occurs an even number of times, and the probability of this is $1/2$, proving the claim.

It follows that the expected number of $(g-1)$-cycles in $H'$ is $M$. Note that the trivial lift $H_0$ of $H$, which has $c^{(k)}_e = 0$ for all $k$ and $e$, consists of two vertex-disjoint copies of $H$, and therefore has $2M$ $(g-1)$-cycles. It follows that there is at least one 2-lift of $H$ with fewer than $M$ $(g-1)$-cycles, contradicting the minimality of $M$. Therefore, $M=0$, so in fact, $H$ has girth at least $g$. This completes the proof of the induction step, proving the theorem.
\end{proof}

\begin{remark} Lemma \ref{lemma:lift} can also be proved by considering a random $r$-uniform, $d$-regular hypergraph on $n$ vertices, for $n$ large. In \cite{cooper}, Cooper, Frieze, Molloy and Reed analyse these using a generalisation of Bollob\'as' configuration model for $d$-regular graphs. It follows from Lemma 2 in \cite{cooper} that if $H$ is chosen uniformly at random from the set of all $r$-uniform, $d$-regular, $n$-vertex, linear hypergraphs (where $r |n$), then
\begin{equation}\label{eq:randomreg} \Pr[\textrm{girth}(H) \geq g] = (1+o(1)) \frac{\exp(-\sum_{l=1}^{g-1} \lambda_l)}{1-\exp(-(\lambda_1+\lambda_2))},\end{equation}
where
$$\lambda_i = \frac{(r-1)^i (d-1)^i}{2i} \quad (i \in \mathbb{N}),$$
so this event occurs with positive probability for sufficiently large $n$, giving an alternative proof of Lemma \ref{lemma:lift}. (We note that the argument of \cite{cooper} can easily be adapted to prove the same statement in the case where $r\, |\, dn$.) \end{remark}

By itself, the proof of Lemma \ref{lemma:lift} implies only that
$$n_r(g,d) \leq \underbrace{
  {{{2^{2\vphantom{h}}}^{2\vphantom{h}}}^{\iddots\vphantom{h}}}^{2^{r^{Cd}}\vphantom{h}}
}_{\text{$g-3\ 2$'s}},$$
where $C$ is an absolute constant --- i.e., tower-type dependence upon $g$. We now proceed to obtain an upper bound which is exponential in $g$.

Consider a $d$-regular graph with girth at least $g$, with the smallest possible number of vertices subject to these conditions. Erd\H{o}s and Sachs \cite{erdossachs} proved that the diameter of such a graph is at most $g$. But a $d$-regular graph with diameter $D$ has at most
$$1+d\sum_{i=0}^{D-1}(d-1)^i$$
vertices (since this is an upper bound on the number of vertices in a ball of radius $D$). This yielded the upper bound (\ref{eq:erdos-sachs}) on the number of vertices in a $d$-regular graph of girth at least $g$ and minimal order.

We need an analogue of the Erd\H{o}s-Sachs argument for hypergraphs.

\begin{lemma}
Let $r,d$ and $g$ be integers with $d \geq 2$ and $r,g \geq 3$. Let $H$ be an $r$-uniform, $d$-regular hypergraph with girth at least $g$, with the smallest possible number of vertices subject to these conditions. Then $H$ cannot contain $r$ vertices every two of which are at distance greater than $g$ from one another.
\end{lemma}
\begin{proof}
Let $H$ be an $r$-uniform, $d$-regular hypergraph with girth at least $g$. Suppose that $H$ contains $r$ distinct vertices $v_1,v_2,\ldots,v_r$ such that $\dist(v_i,v_j) > g$ for all $i \neq j$. We will show that it is then possible to construct an $r$-uniform, $d$-regular hypergraph with girth at least $g$, that has fewer vertices than $H$; this will prove the lemma.

Note that $H$ is linear, since $g \geq 3$. For each $i \in [r]$, let $e^{(1)}_i,e^{(2)}_i,\ldots,e^{(d)}_i$ be the edges of $H$ which contain $v_i$. Let
$$W_i = \bigcup_{k=1}^{d}(e^{(k)}_i \setminus \{v_i\})$$
for each $i \in [r]$. Notice that $|W_i| = d(r-1)$ for each $i$, since the edges $e^{(k)}_i\ (k \in [d])$ are disjoint apart from the vertex $v_i$. Moreover, $W_i \cap W_j = \emptyset$ for all $i \neq j$, since $d(v_i,v_j) > 2$.

Define a new hypergraph $H'$ by taking $H$, deleting $v_1,v_2,\ldots,v_r$ and all the edges containing them, and adding $d(r-1)$ pairwise disjoint edges, each of which contains exactly one vertex from $W_i$ for each $i \in [r]$. (Note that none of these `new' edges were in the original hypergraph $H$, otherwise some $v_{i}$ and $v_{j}$  would have been at distance at most 3 in $H$, a contradiction.) Clearly, $H'$ is $d$-regular. We claim that it is linear. Indeed, if one of the `new' edges shared two vertices with some edge $f \in H$ (say it shares $a \in W_i$ and $b \in W_j$, where $i \neq j$), then there would be a path of length 3 in $H$ from $v_i$ to $v_j$, a contradiction. 

We now claim that $H'$ has girth at least $g$. Suppose for a contradiction that $H'$ has girth at most $g-1$. Let $C$ be a cycle in $H'$ of length $l \leq g-1$. Since $H'$ is linear, we have $l \geq 3$. Let $(f_1,\ldots,f_l)$ be a cyclic ordering of $C$. We split into two cases.

\textit{Case 1.} Suppose that $C$ contains exactly one of the `new' edges (say $f_i$ is a `new' edge). Deleting $f_i$ from $C$ produces a path $P$ of length at most $g-2$ in $H$. We have $|f_{i-1} \cap f_i| = |f_i \cap f_{i+1}|=1$ (since $H'$ is linear); let $f_{i-1}\cap f_i = \{a\}$, and let $f_{i} \cap f_{i+1}=\{b\}$. Note that $a \neq b$. Suppose that $a \in W_p$ and $b \in W_q$. Since $a \neq b$ and $a,b \in f_i$, we must have $p \neq q$, as each `new' edge contains exactly one vertex from each $W_k$. Let $e$ be the edge of $H$ containing both $v_p$ and $a$, and let $e'$ be the edge of $H$ containing both $v_q$ and $b$; adding $e$ and $e'$ to the appropriate ends of the path $P$ produces a path in $H$ of length at most $g$ from $v_p$ to $v_q$, contradicting the assumption that $\textrm{dist}(v_p,v_q) > g$.

\textit{Case 2.} Suppose instead that $C$ contains more than one of the `new' edges. Choose a minimal sub-path $P$ of $C$ which connects two `new' edges. Suppose $P$ connects the new edges $f_i$ and $f_j$, so that $P = (f_i,f_{i+1},\ldots,f_{j-1},f_j)$. Note that $|i-j| \leq (g-1)/2$, so $P$ has length at most $(g+1)/2 \leq g-1$. Let $f_i \cap f_{i+1} = \{a\}$, and suppose $a \in W_p$; let $f_{j-1} \cap f_j = \{b\}$, and suppose $b \in W_q$. Let $e$ be the unique edge of $H$ which contains both $v_p$ and $a$, and let $e'$ be the unique edge of $H$ which contains both $v_q$ and $b$. If $p \neq q$, then we can produce a path in $H$ from $v_p$ to $v_q$ by taking $P$, and replacing $f_i$ with $e$ and $f_j$ with $e'$; this path has length at most $g-1$, contradicting our assumption that $d(v_p,v_q) > g$. If $p=q$, then we can produce a cycle in $H$ by taking $P$, removing $f_i$ and $f_j$, and adding the edges $e$ and $e'$ (which share the vertex $v_p$); this cycle has length at most $g-1$, contradicting our assumption that $H$ has girth at least $g$.

We may conclude that $H'$ has girth at least $g$, as claimed. Clearly, $H'$ has fewer vertices than $H$; this completes the proof.
\end{proof}

This lemma quickly implies an upper bound on the minimal number of vertices in an $r$-uniform, $d$-regular hypergraph of girth at least $g$.

\begin{theorem}
Let $r,d$ and $g$ be integers with $d \geq 2$ and $r,g \geq 3$. There exists an $r$-uniform, $d$-regular hypergraph with girth at least $g$, and at most
$$(r-1)\left(1+d(r-1)\frac{(d-1)^g (r-1)^g -1}{(d-1)(r-1)-1}\right) < 4((d-1)(r-1))^{g+1}$$
vertices. Hence,
$$n_r(g,d) < 4((d-1)(r-1))^{g+1}.$$
\end{theorem}
\begin{proof}
Let $H$ be an $r$-uniform, $d$-regular hypergraph with girth at least $g$, with the smallest possible number of vertices subject to these conditions. Let $\{v_1,v_2,\ldots,v_k\}$ be a set of vertices of $H$ whose pairwise distances are all greater than $g$, with $k$ maximal subject to this condition. By the previous lemma, we have $k < r$. Any vertex of $H$ must have distance at most $g$ from one of the $v_i$'s. For each $i$, the number of vertices of $H$ of distance at most $g$ from $v_i$ is at most
$$1+d(r-1)\sum_{i=0}^{g-1} ((d-1)(r-1))^i = 1+d(r-1) \frac{(d-1)^g (r-1)^g -1}{(d-1)(r-1)-1},$$
and therefore the number of vertices of $H$ is at most
\begin{align*} & k\left(1+d(r-1)\frac{(d-1)^g (r-1)^g -1}{(d-1)(r-1)-1}\right) \\
& \leq (r-1)\left(1+d(r-1)\frac{(d-1)^g (r-1)^g -1}{(d-1)(r-1)-1}\right).\end{align*}
Crudely, we have
$$(r-1)\left(1+d(r-1)\frac{(d-1)^g (r-1)^g -1}{(d-1)(r-1)-1}\right) < 4((d-1)(r-1))^{g+1}$$
for all integers $r,d$ and $g$ with $d \geq 2$ and $r,g \geq 3$, proving the theorem.
\end{proof}
The following corollary is immediate.
\begin{corollary}
\label{corr:lowerbound}
Let $r,d$ and $n$ be positive integers with $d \geq 2$ and $r \geq 3$. There exists an $r$-uniform, $d$-regular hypergraph on at most $n$ vertices, with girth greater than
$$ \frac{\log n - \log 4}{\log(d-1)+\log(r-1)} - 1.$$
Hence,
$$g_{r,d}(n) > \frac{\log n - \log 4}{\log(d-1)+\log(r-1)} - 1.$$
\end{corollary}

Observe that the lower bound in Corollary \ref{corr:lowerbound} differs from the upper bound in Corollary \ref{corr:upperbound} by a factor of (approximately) 2.

For $r,d \geq 3$, we have not been able to improve upon the lower bound in Corollary \ref{corr:lowerbound} for large $n$. As mentioned in the Introduction, in the case of graphs, the bipartite Ramanujan graphs of Lubotzsky, Phillips and Sarnak \cite{lps}, Margulis \cite{margulis} and Morgenstern \cite{morgenstern} provide $d$-regular, $n$-vertex graphs of girth at least
$$(1-o(1))\frac{4}{3} \frac{\log n}{\log (d-1)},$$
for infinitely many $n$, whenever $d-1$ is a prime power. Recall that a finite, connected, $d$-regular graph is said to be {\em Ramanujan} if every eigenvalue $\lambda$ of its adjacency matrix is either `trivial' (i.e. $\lambda = \pm d$), or has $|\lambda| \leq 2\sqrt{d-1}$.

\begin{theorem}[Lubotzsky-Phillips-Sarnak, Margulis, Morgenstern]
\label{thm:ramanujan}
For any odd prime power $p$, there exist infinitely many (bipartite) $(p+1)$-regular Ramanujan graphs $X^{p,q}$. The graph $X^{p,q}$ is a Cayley graph on the group $\textrm{PGL}(2,q)$, so has order $q(q^2-1)$. Moreover, its girth satisfies
$$g(X^{p,q}) \geq \frac{4\log q}{\log p} - \frac{\log 4}{\log p}.$$
\end{theorem}

It is in place to remark that recently, Marcus, Spielman and Srivastava \cite{mss} proved the existence of infinitely many $d$-regular Ramanujan graphs for {\em every} $d \geq 3$. They did this by proving a weakening of a conjecture of Bilu and Linial \cite{b-l} on 2-lifts of Ramanujan graphs, namely, that every $d$-regular Ramanujan graph has a 2-lift whose second-largest eigenvalue is at most $2\sqrt{d-1}$. Their proof uses a beautiful new technique for demonstrating the existence of combinatorial objects, which they call the `method of interlacing polynomials'. (Even more spectacularly, they use this method to prove the Kadison-Singer conjecture, in \cite{ks}.) Being non-constructive, however, their proof does not imply good bounds for the girth problem.

We are able to improve upon the lower bound in Corollary \ref{corr:lowerbound} when $r=3$ and $d=2$, using the following explicit construction, based upon the Ramanujan graphs of Theorem \ref{thm:ramanujan}. Let $G$ be an $n$-vertex, $3$-regular graph of girth $g$. Take any drawing of $G$ in the plane with straight-line edges, and for each edge $e \in E(G)$, let $m(e)$ be its midpoint. Let $H$ be the 3-uniform hypergraph with
\begin{align*}
V(H) &= \{m(e):\ e \in E(G)\},\\
E(H) &= \{\{m(e_1),m(e_2),m(e_3)\}:\\
&\quad\quad e_1,e_2,e_3 \textrm{ are incident to a common vertex of }G\}.\end{align*}

Then the hypergraph $H$ is $2$-regular, and also has girth $g$. Taking $G = X^{2,q}$ (the Ramanujan graph of Theorem \ref{thm:ramanujan}) yields a 3-uniform, 2-regular hypergraph $H$ with 
\begin{align*} g(H) & = g(X^{2,q})\\
& \geq \frac{4\log q}{\log 2} - 2\\
& \geq \frac{4}{3} \frac{\log n}{\log 2}-2\end{align*}
improving upon the bound in Corollary \ref{corr:lowerbound} by a factor of $\tfrac{4}{3} - o(1)$.

The following explicit construction, also based on the Ramanujan graphs of Theorem \ref{thm:ramanujan}, provides $r$-uniform, $d$-regular hypergraphs of girth approximately $2/3$ of  the bound in Corollary \ref{corr:lowerbound}, whenever $d$ is a multiple of $r$. (We thank an anonymous referee of an earlier version of this paper, for pointing out this construction.) 

Suppose $d=rs$ for some $s \in \mathbb{N}$. Let $G$ be a $2(r-1)s$-regular, $n$ by $n$ bipartite graph, with vertex-classes $X$ and $Y$, and girth $g$. Then the edge-set of $G$ may be partitioned into $(r-1)$-edge stars in such a way that each vertex of $G$ is in exactly $rs$ of the stars. (Indeed, by Hall's theorem, we may partition the edge-set of $G$ into $2(r-1)s$ perfect matchings. First, choose $r-1$ of these matchings, and group the edges of these matchings into $n$ $(r-1)$-edge stars with centres in $X$. Now choose $r-1$ of the remaining matchings, and group their edges into $n$ $(r-1)$-edge stars with centres in $Y$. Repeat this process $s$ times to produce the desired partition of $E(G)$ into stars.)

Let $H$ be the $r$-uniform hypergraph whose vertex-set is $X \cup Y$, and whose edge-set is the collection of vertex-sets of these stars; then $H$ is $(rs)$-regular, and has girth at least $g/2$.

If $2(r-1)s-1$ is a prime power, the bipartite Ramanujan graph $X^{p,q}$ (with $p = 2(r-1)s-1$) can be used to supply the graph $G$.
This yields a linear, $r$-uniform, $(rs)$-regular hypergraph with girth $g(H)$ satisfying
\begin{align*}g(H)& \geq \frac{1}{2}\left(\frac{4\log q}{\log (2rs-2s-1)} - \frac{\log 4}{\log (2rs-2s-1)}\right)\\
& \geq \frac{1}{2}\left(\frac{4}{3}\frac{\log n}{\log (2rs-2s-1)} - \frac{\log 4}{\log (2rs-2s-1)}\right)\\
& = \frac{2}{3}\frac{\log n}{\log (2d-2d/r-1)} - \frac{\log 2}{\log (2d-2d/r-1)},\end{align*}
where $d=rs$.

Unfortunately, this lower bound is asymptotically worse than that given by Corollary \ref{corr:lowerbound}, for all values of $r$ and $d$.

\section{Random `Cayley' hypergraphs}
 In this section, we give a construction of random `Cayley' hypergraphs on the symmetric group $S_n$, which have girth $\Omega(n^{1/3})$ with high probability. This is much higher than the girth of a random regular hypergraph on the same number of vertices (which, by (\ref{eq:randomreg}), has girth at most $C(\epsilon)$ with probability at least $1-\epsilon$ for any $\epsilon >0$, where $C(\epsilon)$ is a constant depending on $\epsilon$ alone), though it is still short of the optimal $\Theta(\log |V(H)|)$ in Corollary \ref{corr:lowerbound}. The situation is analogous to the graph case, where random $d$-regular Cayley graphs on appropriate groups have much higher girth than random $d$-regular graphs of the same order (due to the dependency between cycles at different vertices of a Cayley graph).
 
First, we need some more definitions. If $S$ is a set of symbols, a {\em word in} $S$ is a string of the form
$$s_1^{a_1}s_2^{a_2} \ldots s_l^{a_l}$$
where $s_1,\ldots,s_l \in S$ and $a_1,\ldots,a_l \in \mathbb{Z} \setminus \{0\}$. Such a word is said to be {\em cyclically irreducible} if $s_i \neq s_{i+1}$ for all $i \in [l]$, where we define $s_{l+1} := s_1$. Its {\em length} is $\sum_{i=1}^{l}|a_i|$.

\begin{theorem}
\label{thm:sym}
Let $r$ and $n$ be positive integers with $r \geq 3$ and $r|n$. Let $X(n,r)$ be the set of permutations in $S_n$ that consist of $\frac nr$ disjoint $r$-cycles. Choose $d$ permutations $\tau_1,\tau_2,\ldots,\tau_d$ uniformly at random and independently (with replacement) from $X(n,r)$, and let $H$ be the random hypergraph with vertex-set $S_n$ and edge-set
$$\{\{\sigma,\sigma \tau_i, \sigma \tau_i^2,\ldots,\sigma \tau_i^{r-1}\}:\ \sigma \in S_n,\ i \in [d]\}.$$
Then with high probability, $H$ is a linear, $r$-uniform, $d$-regular hypergraph with girth at least
$$c_0 \left(\frac{n}{r(\log(d-1)+\log(r-1))}\right)^{1/3},$$
for some constant $c_0>0$.
\end{theorem}

\begin{remark}
Here, `with high probability' means `with probability tending to 1 as $n \to \infty$'.
\end{remark}
\begin{remark}
(Added 12th July 2017.) We note that in the previous version of this paper, the claimed lower bound on the girth of $H$ in Theorem \ref{thm:sym} was somewhat stronger, viz., $\Omega(\sqrt{n \log n})$. However, our previous proof used a variant of a method in the proof of \cite[Theorem 3]{ghssv}; both that proof and our variant thereof contained a hole, as pointed out by Eberhard in \cite{eberhard}. Here, we repair the hole, albeit giving a slightly weaker bound of $\Omega(n^{1/3})$. The fix uses a very similar method to that of Eberhard in \cite{eberhard}, where a slightly weaker version of \cite[Theorem 3]{ghssv} is proved. More details are given below.
\end{remark}

\begin{proof}
Note that the edges of the form
$$\{\sigma,\sigma \tau_i, \sigma \tau_i^2,\ldots,\sigma \tau_i^{r-1}\}\ (\sigma \in S_n)$$
are simply the left cosets of the cyclic group $\{\textrm{Id},\tau_i,\tau_i^2,\ldots,\tau_i^{r-1}\}$ in $S_n$, so they form a partition of $S_n$. We need two straightforward claims.
\begin{claim}
\label{claim:condition}
With high probability, the following condition holds.
\begin{equation}\label{eq:cond} \tau_1,\ldots,\tau_d \textrm{ satisfy}\quad \tau_i^{k} \neq \tau_j^{l}\quad \textrm{for all distinct }i,j \in [d] \textrm{ and all }k,l \in [r-1].\end{equation}
\end{claim}
\begin{proof}[Proof of claim:]
Let us fix $i,j \in [d]$ with $i <j$, and fix $k,l \in [r-1]$. We shall bound the probability that $\tau_j^{l} = \tau_i^{k}$. We regard $\tau_i$ as fixed, and allow $\tau_j$ to vary. Since $\tau_i$ is a product of $n/r$ disjoint $r$-cycles, $\tau_i^k$ is a product of $n/s$ disjoint $s$-cycles, for some integer $s \ge 2$ that is a divisor of $r$. The set $X(n,s)$ of permutations which consist of $n/s$ disjoint $s$-cycles has cardinality
$$\frac{n!}{(n/s)! s^{n/s}} \geq \frac{n!}{(n/2)!2^{n/2}}$$
(provided $n \geq 4$). Notice that $\tau_j^{l}$ is uniformly distributed over $X(n,s')$, for some $s'$ that depends only on $r$ and $l$. Therefore,
$$\Pr[\tau_i^k = \tau_j^l] \leq \frac{(n/2)!2^{n/2}}{n!}.$$
By the union bound,
\begin{align*}& \Pr[\tau_i^k = \tau_j^l \textrm{ for some }i \neq j \textrm{ and some }k,l \in [r-1]]\\
 & \leq (r-1)^2{d \choose 2} \frac{(n/2)!2^{n/2}}{n!}\\
& \to 0 \quad \textrm{as }n \to \infty,\end{align*}
proving the claim.
\end{proof}

\begin{claim}
\label{claim:linear}
If condition (\ref{eq:cond}) holds, then for all $i \neq j$ and all $\sigma, \pi \in S_n$, the two cosets
$$\{\sigma, \sigma \tau_i, \sigma \tau_i^2,\ldots, \sigma \tau_i^{r-1}\} \mbox{~~and~~} \{\pi, \pi \tau_j, \pi \tau_j^2, \ldots, \pi \tau_j^{r-1}\}$$
have at most one element in common.
\end{claim}
\begin{proof}[Proof of claim:]
Suppose for a contradiction that there are two distinct vertices $v_1,v_2$ with
$$v_1,v_2 \in \{\sigma, \sigma \tau_i, \sigma \tau_i^2,\ldots, \sigma \tau_i^{r-1}\} \cap \{\pi, \pi \tau_j, \pi \tau_j^2, \ldots, \pi \tau_j^{r-1}\}.$$
Then $v_1 = \sigma \tau_i^l = \pi \tau_j^{m}$ and $v_2 = \sigma \tau_i^{l'} = \pi \tau_j^{m'}$, where $l,m,l',m' \in \{0,1,\ldots,r-1\}$ with $l' \neq l$ and $m' \neq m$. Therefore,
$$v_1^{-1}v_2 = \tau_i^{l'-l} = \tau_j^{m'-m},$$
contradicting condition (\ref{eq:cond}).
\end{proof}

Claim \ref{claim:linear} implies that $H$ is a linear hypergraph, provided condition (\ref{eq:cond}) is satisfied. Moreover, $H$ is $d$-regular: every $\sigma \in S_n$ is contained in the edges (cosets)
$$(\{\sigma, \sigma \tau_i,\sigma \tau_i^2,\ldots,\sigma \tau_i^{r-1}\}: i \in [d]),$$
and these $d$ edges are distinct provided condition (\ref{eq:cond}) is satisfied.

Finally, we make the following.
\begin{claim}
\label{claim:girth}
With high probability, $H$ has girth at least
$$c_0 \left(\frac{n}{r(\log(d-1)+\log(r-1))}\right)^{1/3},$$
where $c_0>0$ is an absolute constant.
\end{claim}
\begin{proof}[Proof of claim:]
We may assume that condition (\ref{eq:cond}) holds, so that $H$ is a linear, $d$-regular hypergraph. Let $C$ be a cycle in $H$ of minimum length, and let $(e_1,\ldots,e_l)$ be any cyclic ordering of its edges. Then we have $|e_i \cap e_{i+1}|=1$ for all $i \in [l]$ (where we define $e_{l+1} := e_1$), and by minimality, we have $e_i \cap e_j = \emptyset$ whenever $|i-j| > 1$. Let $e_i \cap e_{i+1} = \{w_i\}$ for each $i \in [l]$. Suppose that $e_i$ is an edge of the form
$$\{\sigma, \sigma \tau_{j_i}, \sigma \tau_{j_i}^2,\ldots, \sigma \tau_{j_i}^{r-1}\}$$
for each $i \in [l]$. Since $e_i \cap e_{i+1} \neq \emptyset$ for each $i \in [l]$, we must have $j_i \neq j_{i+1}$ for all $i \in [l]$ (where we define $j_{l+1}:=j_1$). For each $i \in [l]$, we have $w_i,w_{i+1} \in e_{i+1}$, so $w_{i+1}^{-1}w_{i} = \tau_{j_{i+1}}^{m_i}$ for some $m_i \in [r-1]$. Therefore,
\begin{equation}\label{eq:irreducibleword}\textrm{Id} = (w_l^{-1}w_{l-1}) \ldots (w_{3}^{-1}w_2)(w_{2}^{-1}w_1) (w_1^{-1}w_l)= \tau_{j_{l}}^{m_{l-1}} \ldots \tau_{j_3}^{m_{2}} \tau_{j_{2}}^{m_1} \tau_{j_1}^{m_{l}}.\end{equation}
Since $j_i \neq j_{i+1}$ for all $i \in [l]$, the word on the right-hand side of (\ref{eq:irreducibleword}) is cyclically irreducible, and evaluates to the identity permutation. We must show that the probability of this tends to zero as $n \to \infty$, for an appropriate choice of $l$. We use an argument very similar to (but slightly more involved than) that in \cite{eberhard}, where it is proved that with high probability, the Cayley graph on $S_n$ generated by $d$ random permutations has girth at least $\Omega((n/\log(2d-1))^{1/3})$. We remark that in a previous version of this paper, we used a variant of a proof in \cite{ghssv} of a stronger (claimed) bound, but the latter proof, and our variant thereof, both contain a hole, as pointed out in \cite{eberhard}.

For brevity, we write $m_{i} = a_{i+1}$ for each $i \in [l]$, and we write
$$W := \tau_{j_{l}}^{a_{l}} \ldots \tau_{j_3}^{a_{3}} \tau_{j_{2}}^{a_2} \tau_{j_1}^{a_{1}}$$
for the word on the right-hand side of (\ref{eq:irreducibleword}). For the convenience of the reader, we follow quite closely the structure of the argument in \cite{eberhard}.

We pick $x_1 \in [n]$ arbitrarily, and consider the `trajectory' of $x_1$ under $W$. Formally, we define
\begin{align*}
x_1^0 &= x_1,\\
x_1^i &= \tau_{j_i}^{a_i}(x_1^{i-1}) \quad (1 \leq j \leq l).
\end{align*}

Now suppose that at each `step' $i \in [l]$, when we are about to evaluate $x_1^i$, we examine (or `visit') the $r$-cycle of $\tau_{j_i}$ containing $x_1^{i-1}$, revealing this entire $r$-cycle if it has not been revealed already, but revealing no other information.

We will now show that during this process, with rather high probability, we reveal a new $r$-cycle at every step, i.e.\ we never `revisit' an $r$-cycle that has been `visited' already. 

For each $i \in [l-1]$, we say that $x_1^i$ is {\em good} if the $r$-cycle of $\tau_{j_{i+1}}$ containing $x_1^i$ has not been revealed at any prior step, i.e.\ for each $i' < i$ with $\tau_{j_{i'+1}} = \tau_{j_{i+1}}$, we have $x_1^i$ and $x_1^{i'}$ in different $r$-cycles of $\tau_{j_{i+1}}$. Otherwise, we say that $x_1^i$ is {\em bad}. Note that if $x_1^i$ is bad, and $i' < i$ is such that $\tau_{j_{i'+1}} = \tau_{j_{i+1}}$, then necessarily $i' \leq i-2$, since the word $W$ is cyclically irreducible, and we are assuming that condition (\ref{eq:cond}) holds, so $\tau_{j_i} \neq \tau_{j_{i+1}}$ for all $i$. Note also that the event $\{x_1^i \text{ is bad}\}$ depends only upon the $r$-cycles of $\tau_{j_1},\tau_{j_2},\ldots,\tau_{j_i}$ examined in steps $1,2,\ldots,i$, respectively.

Observe that for each $i \in [l-1]$, we have
\begin{equation}\label{eq:cond-1} \Pr[x_1^i \text{ is bad} \mid x_1^{1},\ldots,x_1^{i-1} \text{ are good}] \leq \frac{(i-1)r}{n-(i-1)r}.\end{equation}
Indeed, condition on the event that $x_1^{1},\ldots,x_1^{i-1}$ are good. Note that $x_1^i$ is in the $r$-cycle of $\tau_{j_i}$ revealed for the first time at step $i$ (here, we use the fact that $x_1^{i-1}$ is good); indeed, $x_1^i$ is the $a_i$th number in this $r$-cycle, after $x_1^{i-1}$. If $x_1^i$ is bad, then $x_1^i$ is also in an $r$-cycle of $\tau_{j_{i+1}}$ revealed (for the first time) at step $i'$, for some $i' <i$ such that $\tau_{j_{i'}} = \tau_{j_{i+1}}$. Since the $r$-cycles revealed at steps $i' < i$ together contain at most $(i-1)r$ numbers, and the $a_i$th number after $x_1^{i-1}$ in the $r$-cycle of $\tau_{j_i}$ revealed (for the first time) at step $i$, is chosen uniformly from a set of at least $n-(i-1)r$ numbers, (\ref{eq:cond-1}) follows.

Now observe that
\begin{equation} \label{eq:cond-2} \Pr[x_1^l = x_1^0 \mid x_1^{1},\ldots,x_1^{l-1} \text{ are good}] \leq \frac{1}{n-(l-1)r}.\end{equation}
Indeed, condition on the event that $x_1^{1},\ldots,x_1^{l-1}$ are good. If $x_1^l = x_1^0$, then $x_1^0$ is in the $r$-cycle of $\tau_{j_{l}}$ revealed for the first time at step $l$ (here, we use the fact that $x_1^{l-1}$ is good); indeed, $x_1^0$ is the $a_l$th number in this $r$-cycle, after $x_1^{l-1}$. Since the $a_l$th number after $x_1^{l-1}$ in the $r$-cycle of $\tau_{j_l}$ revealed (for the first time) at step $l$, is chosen uniformly from a set of at least $n-(l-1)r$ numbers, (\ref{eq:cond-2}) follows.

Combining (\ref{eq:cond-1}) and (\ref{eq:cond-2}), and using a union bound, we obtain
$$\Pr[x_1^{l} = x_1^0] \leq \sum_{i=1}^{l-1} \frac{(i-1)r}{n-(i-1)r} + \frac{1}{n-(l-1)r} \leq \frac{l^2r}{n-lr}.$$

Now we condition on the event $\{x_1^{l} = x_1^0\}$. We pick $x_2$ not in any of the $r$-cycles we have previously exposed, and repeat the argument. In fact, let $m \geq 2$, suppose we have done this $m-1$ times already, condition on the event $E :=\{x_{m'}^l = x_{m'}^0\ \forall m' <m\}$, and suppose that we have just chosen $x_m$ not in any of the $r$-cycles we have previously exposed. Define $x_m^0,\ldots,x_m^l$ as above. For each $i \in [l-1]$, let us say that $x_m^i$ is {\em good} if the $r$-cycle of $\tau_{j_{i+1}}$ containing $x_m^i$ has not been revealed at any prior step (including steps involving $x^{i'}_{m'}$, for $m' < m$). Otherwise, we say that $x_m^i$ is {\em bad}. Arguing as above, we obtain
$$\Pr[x_m^i \text{ is bad} \mid \{x_m^{1},\ldots,x_m^{i-1} \text{ are good}\} \cap E] \leq \frac{(m-1)lr+(i-1)r}{n-(m-1)lr-(i-1)r}$$
for all $i \in [l-1]$, and
$$\Pr[x_m^l =x_m^0 \mid \{x_m^{1},\ldots,x_m^{l-1} \text{ are good}\} \cap E] \leq \frac{1}{n-(m-1)lr-(l-1)r},$$
so by a union bound, we obtain
$$\Pr[x_m^l =x_m^0 \mid E] \leq \frac{ml^2r}{n-mlr}.$$
In order to have $W = \text{Id}$, we must have $x_{s}^l = x_{s}^0$ for all $s \in [m]$. Thus,
$$\Pr[W = \text{Id}] \leq \prod_{s=1}^{m} \Pr[x_{s}^l = x_{s}^0 \mid \{x_{t}^{l} = x_t^{0}\ \forall t < s\}] \leq \left(\frac{ml^2r}{n-mlr}\right)^m.$$
Choosing $m = \lfloor n/(4l^2r)\rfloor$ yields
$$\Pr[W = \text{Id}] \leq 2^{-n/(4l^2r)+1}.$$

The number of choices for the word on the right-hand side of (\ref{eq:irreducibleword}) is at most $(d-1)^{l} (r-1)^l$. (By taking a cyclic shift if necessary, we may assume that $j_2 \neq d$, so there are at most $d-1$ choices for $j_2$, and at most $d-1$ choices for all subsequent $j_i$; there are clearly at most $r-1$ choices for each $m_i$.) Hence, the probability that there exists such a word which evaluates to the identity permutation is at most
$$(d-1)^{l} (r-1)^l 2^{-n/(4l^2r)+1}.$$
To bound the probability that $H$ has a cycle of length less than $g$, we need only sum the above expression over all $l < g$:
\begin{align*} \Pr[\textrm{girth}(H) < g] & \leq \sum_{l=3}^{g-1} (d-1)^{l} (r-1)^l 2^{-n/(4l^2r)+1}\\
& < 2(d-1)^{g} (r-1)^g 2^{-n/(4g^2r)}.\end{align*}
In order for the right-hand side to tend to zero as $n \to \infty$, it suffices to choose
$$g \leq c_0 \left(\frac{n}{r(\log(d-1)+\log(r-1))}\right)^{1/3},$$
for some absolute constant $c_0 >0$. This completes the proof of Claim \ref{claim:girth}, and thus proves Theorem \ref{thm:sym}.
\end{proof}
\end{proof}

\section{Conclusion and open problems}\label{sec:open}
Our best (general) upper and lower bounds on the function $g_{r,d}(n)$ differ approximately by a factor of 2:
$$(1+o(1))\frac{\log n}{\log(d-1)+\log(r-1)} \leq g_{r,d}(n) \leq (2+o(1))\frac{\log n}{\log(r-1)+\log(d-1)}.$$
It would be of interest to narrow the gap, possibly by means of an explicit algebraic construction {\em \`a la} Ramanujan graphs.

In \cite{ghssv}, Gamburd, Hoory, Shahshahani, Shalev and Vir\'ag conjecture that with high probability, a Cayley graph on $S_n$ generated by $d$ random permutations has girth at least $\Omega(n \log n)$; one may compare this to the best known lower bound, which is $\Omega(n^{1/3})$, in \cite{eberhard}. We believe that the random hypergraph of Theorem \ref{thm:sym} also has girth $\Omega(n \log n)$, with high probability.

In this paper, we considered a very simple and purely combinatorial notion of girth in hypergraphs, but other notions appear in the literature, for example using the language of simplicial topology, such as in \cite{lub:mesh, goff}. A different combinatorial definition was introduced by Erd\H{o}s in \cite{erdos:sts}. Define the $(-2)$-girth of a $3$-uniform hypergraph as the smallest integer $g \geq 4$ such that there is a set of $g$ vertices spanning at least $g-2$ edges. Erd\H{o}s conjectured in \cite{erdos:sts} that there exist Steiner Triple Systems with arbitrarily high $(-2)$-girth; this question remains wide open (see for example \cite{beezer}), and seems very hard.  In view of this, we raise the following.

\begin{question}
\label{question:quad}
Is there a constant $c>0$ such that there exist $n$-vertex $3$-uniform hypergraphs with $cn^2$ edges and arbitrarily high $(-2)$-girth?
\end{question}

Note that Erd\H{o}s' conjecture on Steiner Triple Systems, if true, would imply a positive answer for every $c<\frac 16$. This is clearly tight, since an $n$-vertex, 3-uniform hypergraph with at least $n^2/6$ edges cannot be linear,\footnote{If $H$ is a linear, $n$-vertex, 3-uniform hypergraph, then any pair of vertices is contained in at most one edge of $H$, so double-counting the number of times a pair of vertices in contained in an edge of $H$, we obtain $3e(H) \leq {n \choose 2}$.} and therefore has $(-2)$-girth 4. 

We turn briefly to some variants of Erd\H{o}s' definition. The celebrated $(6,3)$-theorem of Ruzsa and Szemer\'edi~\cite{ruzs:sz} states that if $H$ is an $n$-vertex, $3$-uniform hypergraph in which no 6 vertices span 3 or more edges, then $H$ has $o(n^2)$ edges. Therefore, if we define the $(-3)$-girth of a 3-uniform hypergraph to be the smallest integer $g \geq 6$ such that there exists a set of $g$ vertices spanning at least $g-3$ edges,\footnote{The condition $g \geq 6$ is necessary to avoid triviality: if we replaced it with $g \geq 5$, then a 3-uniform hypergraph would have $(-3)$-girth 5 unless it consisted of isolated edges.} then an $n$-vertex, 3-uniform hypergraph with $(-3)$-girth at least 7 has $o(n^2)$ edges. Hence, the analogue of Question \ref{question:quad} for $(-3)$-girth has a negative answer. On the other hand, if we define the $(-1)$-girth of a 3-uniform hypergraph to be the smallest integer $g$ such that there exists a set of $g$ vertices spanning at least $g-1$ edges, it can be shown that the maximum number of edges in an $n$-vertex, 3-uniform hypergraph with $(-1)$-girth at least $g$, is $n^{2+\Theta(1/g)}.$ 

\section{Acknowledgment}
N. L. wishes to thank Shlomo Hoory for many years of joint pursuit of the girth problem~\cite{hoory}. Quite a few of the ideas of this present paper can be traced back to that joint research effort.

\end{document}